\documentclass[11pt,reqno]{amsart}
\usepackage[utf8]{inputenc}
\usepackage{amscd,amssymb,amsmath,amsthm}
\usepackage[arrow,matrix]{xy}
\usepackage{graphicx}
\usepackage{epstopdf}
\usepackage{color}
\usepackage{tikz}
\usetikzlibrary {positioning}

\newtheorem{thm}{Theorem}
\newtheorem{defn}{Definition}
\newtheorem{lemma}{Lemma}
\newtheorem{pro}{Proposition}
\newtheorem{rk}{Remark}

\makeatletter
\@namedef{subjclassname@2020}{%
  \textup{2020} Mathematics Subject Classification}
\makeatother

\numberwithin{equation}{section} 

\begin{document}
\title [Description of  trajectories of an evolution operator generated by mosquito population]
{Description of  trajectories of an evolution operator generated by mosquito population}

\author{Z.S. Boxonov}

\address{V.I.Romanovskiy Institute of Mathematics of Uzbek Academy of Sciences}

\begin{abstract} In this paper, we study discrete-time dynamical systems generated by evolution operator of mosquito population. An invariant set is found and a Lyapunov function with respect to the operator is constructed in this set. Using the Lyapunov function, the global attraction of a fixed point is proved. Moreover, we give some biological interpretations of our results.

\end{abstract}


 \maketitle

\section{\bf Introduction}

Mosquito control manages the population of mosquitoes to reduce their damage to human health and economies. Mosquito control is a vital public-health practice throughout the world and especially in the tropics because mosquitoes spread many diseases, such as malaria and the Zika virus.

Mathematical models have been formulated in the literature \cite{J}, \cite{J.Li}, \cite{RV} to study the interactive dynamics and control of the wild and sterile mosquito populations. Many concrete models of mathematical biology described by corresponding non-linear evolution operator (see \cite{D}, \cite{Rpd}). Since there is no any general theory of non-linear operators, for each concrete such operator one has to use a specific method of investigation.

In this paper, we consider a wild mosquito population. For the simplified stage-structured mosquito population, we group the three aquatic stages into the larvae class by $x$, and divide the mosquito population into the larvae class and the adults, denoted by $y$. We assume that the density dependence exists only in the larvae stage \cite{J.Li}.

We let the birth rate, that is, the oviposition rate of adults be $\beta$; the rate of emergence from larvae to adults be a function of the larvae with the form of $\alpha(1-k(x))$,
where $\alpha>0$ is the maximum emergence rete, $0\leq k(x)\leq 1$, with $k(0)=0, k'(x)>0$, and $\lim\limits_{x\rightarrow \infty}k(x)=1$, is the functional response due to the intraspecific competition \cite{J}. We further assume a functional response for $k(x)$, as in \cite{J}, in the form $$k(x)=\frac{x}{1+x}.$$ We let the death rate of larvae be a linear function, denoted by $d_{0}+d_{1}x$, and
the death rate of adults be constant, denoted by $\mu$. The interactive dynamics for the wild mosquitoes are governed by the operator $W:\mathbb{R}^{2}\rightarrow \mathbb{R}^{2}$ defined as
\begin{equation}\label{systema}
W: \left\{%
\begin{array}{ll}
    x'=\beta y-\frac{\alpha x}{1+x}-(d_{0}+d_{1}x)x+x,\\[3mm]
    y'=\frac{\alpha x}{1+x}-\mu y+y
\end{array}%
\right.\end{equation}
where $\alpha >0, \beta >0, \mu >0,\ d_{0}\geq0,\ d_{1}\geq0.$ 

 If for each possible state $\textbf{z}=(x,y)\in\mathbb{R}^{2}$ describing the current generation, the state $\textbf{z}'=(x',y')\in\mathbb{R}^{2}$ is uniquely defined as $\textbf{z}'=W(\textbf{z})$ the map $W:\mathbb{R}^{2}\rightarrow \mathbb{R}^{2}$ called the evolution operator \cite{L}, \cite{Rpd}.

The main problem for a given operator $W$ and arbitrarily initial point $\textbf{z}^{(0)}=(x^{(0)},y^{(0)})\in\mathbb{R}^{2}$  is to describe the limit points of the trajectory $\{\textbf{z}^{(m)}\}_{m=0}^{\infty}$, where $\textbf{z}^{(m)}=W^m(\textbf{z}^{(0)}).$

The dynamics of the operator (\ref{systema}) were studied in detail under conditions $\beta=\mu,\ d_{0}=d_{1}=0$  in \cite{BR1} and under conditions $\beta\neq\mu,\ d_{0}=d_{1}=0$ in \cite{BR2}.
In \cite{B} under condition $\beta<\mu\left(1+\frac{d_{0}}{\alpha}\right)$ we showed that any trajectory converges to unique fixed point. Thus the remaining case is $\beta\geq\mu\left(1+\frac{d_{0}}{\alpha}\right)$. Which was not studied yet. In this paper, we study the case $\beta\geq\mu\left(1+\frac{d_{0}}{\alpha}\right)$ and give an analysis of limit points of trajectory generated by the operator $W$ given by (\ref{systema}).

\section{\bf Preliminaries and Known Results}

Consider systems of difference equations of the form
\begin{equation}\label{xy}
\left\{%
\begin{array}{ll}
    x_{n+1}=f(x_n, y_n),\\[2mm]
    y_{n+1}=g(x_n, y_n),
\end{array} n=0,1,2,... %
\right.\end{equation}
where $f$ and $g$ are given functions and the initial condition $(x_0, y_0)$ comes from some considered set in the intersection of the domains of $f$ and $g$.

Let $\mathfrak{R}$ be a subset of $\mathbb{R}^2$ with nonempty interior, and let $W:\mathfrak{R}\rightarrow\mathfrak{R}$ be a continuous operator. When the function $f(x, y)$ is increasing in $x$ and decreasing in $y$ and the function $g(x, y)$ is decreasing in $x$ and increasing in $y$, the system (\ref{xy}) is called competitive. When the function $f(x, y)$ is increasing in $x$ and increasing in $y$ and the function $g(x, y)$ is increasing in $x$ and increasing in $y$, the system (\ref{xy}) is called cooperative. An operator $W$ that corresponds to the system (\ref{xy}) is defined as $W(x, y)=\left(f(x, y), g(x, y)\right)$. Competitive and cooperative operators, which are called monotone operators, are defined similarly. Strongly cooperative systems of difference equations or operators are those for which the functions $f$ and $g$ are coordinate-wise strictly monotone \cite{BBK}.

Let $\mathbb{R}_{+}^{2}=\{(x,y): x,y\in\mathbb{R}, x\geq0, y\geq0\}.$

Note that the operator $W$ given by (\ref{systema}) well defined on $\mathbb{R}^2_{+}\setminus \{(x, y): x=-1\}$. But to define a discrete-time dynamical system of continuous operator as population we assume $x\geq0$ and $y\geq0$. Therefore, we choose parameters of the operator $W$ to
guarantee that it maps $\mathbb{R}^2_{+}$ to itself.
\begin{lemma}(see \cite{BR1}) If
\begin{equation}\label{parametr}
\alpha>0,\ \beta>0,\ 0<\mu\leq1,\ d_{0}>0,\ \alpha+d_{0}\leq1,\ d_{1}=0
\end{equation}
then the operator (\ref{systema}) maps the set $\mathbb{R}_{+}^{2}$ to itself.
\end{lemma}
In this case the system (\ref{systema}) becomes
\begin{equation}\label{syst}
W_{0}:\left\{%
\begin{array}{ll}
    x'=\beta y-\left(\frac{\alpha }{1+x}+d_{0}-1\right)x,\\[3mm]
    y'=\frac{\alpha x}{1+x}+(1-\mu)y.
\end{array}%
\right.\end{equation}

A point $\textbf{z}\in\mathbb{R}_{+}^{2}$ is called a fixed point of $W_{0}$ if $W_{0}(\textbf{z})=\textbf{z}$.
\begin{defn}\label{d1}
(see\cite{D}) A fixed point $\textbf{z}$ of the operator $W_0$ is called
hyperbolic if its Jacobian $J$ at $\textbf{z}$ has no eigenvalues on the
unit circle.
\end{defn}
\begin{defn}\label{d2}
(see\cite{D}) A hyperbolic fixed point $\textbf{z}$ called:
\begin{itemize}
  \item [(i)] attracting if all the eigenvalues of the Jacobi matrix $J(\textbf{z})$ are less than 1 in absolute value;
  \item [(ii)] repelling if all the eigenvalues of the Jacobi matrix $J(\textbf{z})$ are greater than 1 in absolute value;
  \item [(iii)] a saddle otherwise.
\end{itemize}
\end{defn}

The following proposition is known about the fixed points of the operator $W_0$ and their type.
\begin{pro}\label{type}(\cite{B}) The type of the fixed points for (\ref{syst}) are as follows:
\begin{itemize}
  \item[i)] if $\beta\leq\mu\left(1+\frac{d_{0}}{\alpha}\right)$,  then the operator (\ref{syst}) has unique fixed point $(0,0)$, the point
  $$(0,0)=\left\{\begin{array}{lll}
  attracting, \ \ \mbox{if} \ \ \beta<\mu\left(1+\frac{d_{0}}{\alpha}\right) \\[2mm]
  non-hyperbolic, \ \ \mbox{if} \ \  \beta=\mu\left(1+\frac{d_{0}}{\alpha}\right).
  \end{array}\right.$$
\item[ii)]if \ $\beta>\mu\left(1+\frac{d_{0}}{\alpha}\right)$, then the operator has two fixed points $(0,0)$, $(x^*,y^*)$, and the point $(x^*,y^*)$ is attracting,  the point
  $$(0,0)=\left\{\begin{array}{lll}
  repelling, \ \ \mbox{if} \ \ \beta>\mu\left(1+\frac{d_{0}}{\alpha}\right)+\alpha^* \\[2mm]
  saddle, \ \ \mbox{if} \ \ \mu\left(1+\frac{d_{0}}{\alpha}\right)<\beta<\mu\left(1+\frac{d_{0}}{\alpha}\right)+\alpha^* \\[2mm]
  non-hyperbolic, \ \ \mbox{if} \ \  \beta=\mu\left(1+\frac{d_{0}}{\alpha}\right)+\alpha^*.
  \end{array}\right.$$
  \end{itemize}
where $$\alpha^*=\frac{1}{\alpha}\left(4-2(\alpha+\mu+d_{0})\right),$$
\begin{equation}\label{x*y*}x^*=\frac{\alpha(\beta-\mu)}{\mu d_{0}}-1,\ y^*=\frac{\alpha(\beta-\mu)-\mu d_{0}}{\mu(\beta-\mu)}.\end{equation}
\end{pro}

The following theorem  describes the trajectory of any initial point $\textbf{z}^{(0)}=(x^{(0)}, y^{(0)})$ in $\mathbb{R}^2_{+}$.

\begin{thm}\label{pr} (\cite{B}) For the operator $W_{0}$ given by (\ref{syst}) (i.e. under condition (\ref{parametr}))
and for any initial point $(x^{(0)}, y^{(0)})\in \mathbb R^2_+$ the following hold:
\begin{itemize}
\item[(i)] If $y^{(n)}>\frac{\alpha}{\mu}$ for any natural number $n$ then $\lim\limits_{n\to \infty}x^{(n)}=+\infty, \ \lim\limits_{n\to \infty}y^{(n)}=\frac{\alpha}{\mu};$
\item[(ii)] If there exists $n_{0}$ number such that  $y^{(n_{0})}\leq\frac{\alpha}{\mu}$ and $\beta<\mu\left(1+\frac{d_{0}}{\alpha}\right)$ then $\lim\limits_{n\to \infty}x^{(n)}=0, \ \lim\limits_{n\to \infty}y^{(n)}=0,$
\end{itemize}
where $(x^{(n)}, y^{(n)})=W_0^n(x^{(0)}, y^{(0)})$, with $W_{0}^n$ is $n$-th iteration of $W_{0}$.
\end{thm}

Denote
$$\Omega=\left\{(x,y)\in\mathbb{R}_+^{2}, \ \ 0\leq x\leq \frac{\alpha\beta}{\mu d_{0}}, \ \ 0\leq y\leq \frac{\alpha}{\mu}\right\}.$$

From the Theorem \ref{pr} one can get that for any initial point $\textbf{z}^{(0)}=(x^{(0)}, y^{(0)})$ taken from $ \mathbb R^2_+,$ in a trajectory, if all $y^{(n)}$ are greater than $\frac{\alpha}{\mu}$, then the limit point is $(+\infty, \frac{\alpha }{\mu })$. If $y^{(n_0)}$ is less than $\frac{\alpha}{\mu}$ for some $n_0$, then $y^{(n)}$ is less than $\frac{\alpha}{\mu}$ for all $n>n_0.$ One can see that if $y^{(n)}\leq\frac{\alpha}{\mu}$ then $x^{(n)}\leq \frac {\alpha\beta}{\mu d_{0}}$ \cite{B}. Therefore, it suffices to study the dynamics of the operator $W_{0}$ on the set $\Omega.$

\section{\bf Main Part}

Let us consider the dynamics of the operator $W_{0}$ given by (\ref{syst}) in the sets $\Omega$ under condition $\beta\geq\mu(1+\frac{d_0}{\alpha}).$
We divide the set $\Omega$ into four parts as follows.
$$\Omega_1=\left\{(x,y)\in\Omega, \ \ 0\leq x\leq x^*, \ \ 0\leq y\leq y^*\right\},$$
$$\Omega_2=\left\{(x,y)\in\Omega, \ \ x^*\leq x\leq\frac{\alpha\beta}{\mu d_{0}}, \ \ y^*\leq y\leq \frac{\alpha}{\mu}\right\},$$
$$\Omega_3=\left\{(x,y)\in\Omega, \ \ x^*< x\leq\frac{\alpha\beta}{\mu d_{0}}, \ \ 0\leq y\leq y^*\right\},$$
$$\Omega_4=\left\{(x,y)\in\Omega, \ \ 0\leq x<x^*, \ \ y^*< y\leq \frac{\alpha}{\mu}\right\},$$
where $x^*, y^*$ is the fixed point defined by $(\ref{x*y*})$.

\subsection{Invariant  sets.}\

A set $A$ is called invariant with respect to $W_{0}$ if $W_{0}(A)\subset A$.

\begin{lemma} The sets $\Omega_1, \ \Omega_2$ and $\Omega$ are invariant with respect to $W_{0}.$
\end{lemma}
\begin{proof}
\textbf{(1)} Let $(x,y)\in\Omega_1$, i.e., $0\leq x\leq x^*, 0\leq y\leq y^*.$ Then
$$x'-x^*=\beta y+x\left(1-d_0-\frac{\alpha}{1+x}\right)-x^*\leq\beta y^*+x^*\left(1-d_0-\frac{\alpha}{1+x}\right)-x^*=$$
$$=\beta y^*-d_0x^*-\frac{\alpha x^*}{1+x}\leq\beta y^*-d_0x^*-\frac{\alpha x^*}{1+x^*}=\beta y^*-d_0x^*-\mu y^*=$$
$$=(\beta-\mu)\cdot\frac{\alpha(\beta-\mu)-\mu d_0}{\mu(\beta-\mu)}-d_0\cdot\left(\frac{\alpha(\beta-\mu)}{\mu d_0}-1\right)=0.$$
$$y^*-y'=y^*-(1-\mu)y-\frac{\alpha x}{1+x}\geq y^*-(1-\mu)y^*-\frac{\alpha x^*}{1+x^*}=\mu y^*-\frac{\alpha x^*}{1+x^*}=0.$$
Thus $(x',y')\in\Omega_1$, i.e., $W_0(\Omega_1)\subset\Omega_1.$

\textbf{(2)} Let $(x,y)\in\Omega_2$, i.e., $x^*\leq x\leq\frac{\alpha\beta}{\mu d_{0}}, y^*\leq y\leq \frac{\alpha}{\mu}.$ Then
$$x^*-x'=x^*-\beta y-x\left(1-d_0-\frac{\alpha}{1+x}\right)\leq x^*-\beta y^*-x^*\left(1-d_0-\frac{\alpha}{1+x}\right)=$$
$$=d_0x^*+\frac{\alpha x^*}{1+x}-\beta y^*\leq d_0x^*+\frac{\alpha x^*}{1+x^*}-\beta y^*=d_0x^*+\mu y^*-\beta y^*=0$$
$$y'-y^*=\frac{\alpha x}{1+x}+(1-\mu)y-y^*\geq\frac{\alpha x^*}{1+x^*}+(1-\mu)y^*-y^*=0.$$
So $(x',y')\in\Omega_2$, i.e., $W_0(\Omega_2)\subset\Omega_2.$

\textbf{(3)} Let $(x,y)\in\Omega$, i.e., $0\leq x\leq \frac{\alpha\beta}{\mu d_{0}},\ 0\leq y\leq \frac{\alpha}{\mu}.$ Then
$$x'-\frac{\alpha\beta}{\mu d_{0}}=\beta y+x\left(1-d_{0}-\frac{\alpha}{1+x}\right)-\frac{\alpha\beta}{\mu d_{0}}\leq\frac{\alpha\beta}{\mu}+\frac{\alpha\beta}{\mu d_{0}}\left(1-d_{0}-\frac{\alpha}{1+x}\right)-\frac{\alpha\beta}{\mu d_{0}}=$$
$$=\frac{\alpha\beta}{\mu}-\frac{\alpha\beta}{\mu d_{0}}\left(d_{0}+\frac{\alpha}{1+x^*}\right)=-\frac{\alpha^2\beta}{\mu d_0(1+x)}<0.$$

$$\frac{\alpha}{\mu}-y'=\frac{\alpha}{\mu}-(1-\mu)y-\frac{\alpha x}{1+x}\geq \frac{\alpha}{\mu}-(1-\mu)\frac{\alpha}{\mu}-\frac{\alpha x}{1+x}=\frac{\alpha}{1+x}>0.$$\\
Thus $(x',y')\in\Omega$, i.e., $W_0(\Omega)\subset\Omega.$
\end{proof}

\subsection{Periodic  points.}\

A point $\textbf{z}$ in $W_0$ is called periodic point of $W_0$ if there exists $p$ so that $W_0^{p}(\textbf{z})=\textbf{z}$. The smallest positive integer $p$ satisfying $W_0^{p}(\textbf{z})=\textbf{z}$ is called the prime period or least period of the point $\textbf{z}.$

\begin{thm}\label{perT} For $p\geq 2$ the operator (\ref{syst}) does not have any $p$-periodic point in the set $\mathbb{R}^{2}_{+}.$
\end{thm}
\begin{proof}
Let us first describe periodic points with $p=2$ on $\mathbb{R}^{2}_{+},$ in this case the equation is $W_{0}(W_{0}(\textbf{z}))=\textbf{z}.$
That is
\begin{equation}\label{per2}
\left\{
  \begin{array}{ll}
   x=\beta\left(\frac{\alpha x}{1+x}+(1-\mu)y\right)-\frac{\alpha\left(\beta y-\frac{\alpha x}{1+x}+(1-d_0)x\right)}{1+\beta y-\frac{\alpha x}{1+x}+(1-d_0)x}+\left(1-d_0\right)\left(\beta y-\frac{\alpha x}{1+x}+(1-d_0)x\right), \\[3mm]
   y=\frac{\alpha(\beta y-\frac{\alpha x}{1+x}+(1-d_0)x)}{1+\beta y-\frac{\alpha x}{1+x}+(1-d_0)x}+(1-\mu)\left(\frac{\alpha x}{1+x}+(1-\mu)y\right).\\
  \end{array}
\right.
\end{equation}
By adding first and second equations of the system (\ref{per2}) one can find $y:$
\begin{equation}\label{yper}
y=\frac{x\left(d_0(2-d_0)(1+x)-\alpha(\beta-\mu+d_0)\right)}{(1+x)\left(\beta(2-\mu-d_0)-\mu(2-\mu)\right)}.
\end{equation}
If $\beta\geq\mu(1+\frac{d_0}{\alpha})$, then  $\beta(2-\mu-d_0)>\mu(2-\mu).$ Clearly, if $d_0(2-d_0)\geq\alpha(\beta-\mu+d_0)$ then $x\geq0 \ \Rightarrow \ y\geq0$.  If $d_0(2-d_0)<\alpha(\beta-\mu+d_0)$ then $x\geq\frac{\alpha(\beta-\mu+d_0)}{d_0(2-d_0)}-1 \ \Rightarrow \ y\geq0$.

We obtain the following equation by substituting the $y$ defined in (\ref{yper}) into the second equation of the system (\ref{per2}):
\begin{equation}\label{kubik}
x\left(B_1x^3+B_2x^2+B_3x+B_4\right)=0,
\end{equation}
where
\begin{equation}\label{koef}
\begin{array}{lllllllll}
B_1=A_0A_3,\\[3mm]
B_2=2A_0A_3+A_1A_3+A_0A_4-\alpha A_0,\\[3mm]
B_3=A_0A_3+A_1A_3+A_3+A_0A_4+A_1A_4-2\alpha A_0-A_2,\\[3mm]
B_4=A_3+A_4-\alpha A_0-A_2,\\[3mm]
A_0=1-d_0+\frac{\beta d_0(2-d_0)}{\beta(2-\mu-d_0)-\mu(2-\mu)},\\[3mm]
A_1=1-\alpha-\frac{\alpha\beta (\beta-\mu+d_0)}{\beta(2-\mu-d_0)-\mu(2-\mu)},\\[3mm]
A_2=-\alpha^2\left(1+\frac{\beta (\beta-\mu+d_0)}{\beta(2-\mu-d_0)-\mu(2-\mu)}\right),\\[3mm]
A_3=\frac{\mu(2-\mu)d_0(2-d_0)}{\beta(2-\mu-d_0)-\mu(2-\mu)},\\[3mm]
A_4=-\left(\alpha(1-\mu)+\frac{\alpha\mu(2-\mu)(\beta-\mu+d_0)}{\beta(2-\mu-d_0)-\mu(2-\mu)}\right).\\[3mm]
\end{array}
\end{equation}
Obviously, the fixed points $x=0$ and $x=x^*$ (see (\ref{x*y*})) are solutions of the equation (\ref{kubik}). Thus, dividing the left side of the equation (\ref{kubik}) by $x(x-x^*)$, we get the following quadratic equation:
\begin{equation}\label{kvadrat}
B_1x^2+(B_2+B_1x^*)x+B_3+B_2x^*+B_1x{^*}^2=0.
\end{equation}

Denote
$$B_0=\frac{\alpha(\beta-\mu+d_0)}{d_0(2-d_0)}-1.$$
If $B_0\leq0$ then
\begin{equation}\label{b0<0}
(2-\mu)(2-d_0)-\alpha(2+\beta-\mu)\geq0.
\end{equation}
Indeed,
$(2-\mu)(2-d_0)-\alpha(2+\beta-\mu)=(2-\mu)(2-d_0)-2\alpha-\alpha(\beta-\mu)\geq(2-\mu)(2-d_0)-2\alpha-d_0(2-d_0)+\alpha d_0=(2-d_0)(2-\mu-d_0-\alpha)\geq0.$

We write the coefficients of the equation (\ref{kvadrat}) as following:
\begin{equation}\label{koef+}
\begin{array}{lllllllll}
B_1=A_0A_3,\\[3mm]
B_2+B_1x^*=\frac{\mu d_0(\beta-\mu)(2-\mu)(2-d_0)\left((2-\mu)(2-d_0)-\alpha(2+\beta-\mu)\right)}{\left(\beta(2-\mu-d_0)-\mu(2-\mu)\right)^2},\\[3mm]
B_3+B_2x^*+B_1x{^*}^2=\frac{\mu d_0\left((2-\mu)(2-d_0)-\alpha(2+\beta-\mu)\right)}{\beta(2-\mu-d_0)-\mu(2-\mu)}.\\[3mm]
\end{array}
\end{equation}

It can be seen from (\ref{parametr}) and (\ref{b0<0}) that if $B_ 0\leq 0$, then the coefficients of the equation (\ref{kvadrat}) are positive. Therefore, there is no any positive solution of the equation (\ref{kvadrat}).

Let's consider the case $B_0>0$. For $y$ to be non-negative, it must be $x\geq B_0$. So, we solve the equation (\ref{kvadrat}) in  $[B_0, +\infty)$. In the equation (\ref{kvadrat}) we put  $x+B_0$ instead of $x$ and get:
\begin{equation}\label{B0kvadrat}
B_1x^2+(B_2+B_1x^*+2B_0B_1)x+B_3+B_2(x^*+B_0)+B_1(x{^*}^2+B_0^2+B_0x^*)=0.
\end{equation}
Need to find positive solutions (\ref{B0kvadrat}) instead of looking for all solutions (\ref{kvadrat}) in $[B_0, +\infty)$.

If $B_0>0$, then
\begin{equation}\label{>0}
(\beta-\mu)(1-d_0)>d_0.
\end{equation}
Indeed, $(\beta-\mu)(1-d_0)-d_0\geq\alpha(\beta-\mu)-d_0> d_0(2-d_0-\alpha)-d_0=d_0(1-d_0-\alpha)\geq0.$

If $(\beta-\mu)(1-d_0)>d_0$, then
\begin{equation}\label{>00}
(\beta-\mu)^2(1-d_0)>(\beta-\mu+1)d_0^2.
\end{equation}
Indeed, $(\beta-\mu)^2(1-d_0)-(\beta-\mu+1)d_0^2>(\beta-\mu)d_0-(\beta-\mu)d_0^2-d_0^2=(\beta-\mu)d_0(1-d_0)-d_0^2>d_0^2-d_0^2=0.$

We write the coefficients of the equation (\ref{B0kvadrat}) in the following form:
\begin{equation}\label{B0koef+}
\begin{array}{llll}
B_1=A_0A_3,\\[3mm]
B_2+B_1x^*+2B_0B_1=B_0B_1+\frac{\mu(2-\mu)\left((d_0(2-d_0)+\alpha((\beta-\mu)(1-d_0)-d_0)\right)}{\beta(2-\mu-d_0)-\mu(2-\mu)},\\[3mm]
B_3+B_2(x^*+B_0)+B_1(x{^*}^2+B_0^2+B_0x^*)=\frac{\alpha\mu(2+\beta-\mu)(2-d_0)(1-\mu)d_0^2}{d_0(2-d_0)(\beta(2-\mu-d_0)-\mu(2-\mu))}+\\[3mm]
+\frac{\alpha\mu\alpha(2-\mu)((\beta-\mu)^2(1-d_0)-(\beta-\mu+1)d_0^2)}{d_0(2-d_0)(\beta(2-\mu-d_0)-\mu(2-\mu))}.\\[3mm]
\end{array}
\end{equation}
From (\ref{parametr}), (\ref{>0}), (\ref{>00}) one can see that all coefficients of the equation (\ref{B0kvadrat}) are positive, i.e. there is no any positive solution. Thus, there is no positive solution for the equation (\ref{B0kvadrat}).

Finally, the operator $W_0$ does not have two periodic points in the set $\mathbb{R}^{2}_{+}$.

The Jacobian matrix of the operator $W_0$ has a sign configuration at all point $\textbf{z}$
$$sign(J_{W_0})=\left(%
\begin{array}{cc}
  + & + \\
  + & + \\
\end{array}%
\right).$$
Therefore, the system (\ref{syst}) is strictly cooperative.
Since system (\ref{syst}) is strictly cooperative, Sharkovskii’s ordering holds for periodic points \cite{WJ} and so the no existence of a two-period point would imply the no presence of periodic points of all other periods.
\end{proof}

\subsection{Lyapunov functions.}\

By using Lyapunov functions, one can handle with $\omega$-limits point. Recall the definition of a Lyapunov function.
\begin{defn}\label{Lyap} A continuous functional $\varphi: \Omega\rightarrow\mathbb{R}$ is called a Lyapunov function for a operator $W$ if $\varphi(W(\textbf{z}))\leq\varphi(\textbf{z})$ for all $\textbf{z}$ (or $\varphi(W(\textbf{z}))\geq\varphi(\textbf{z})$ for all $\textbf{z}$).
\end{defn}
\begin{pro}\label{Lpro} For any $\beta, \mu$ (under condition (\ref{parametr})), the function $$\varphi(\textbf{z})=\mu x+\beta y$$ is Lyapunov function for the operator $W_0$ defined by (\ref{syst}).
\end{pro}
\begin{rk} If $\beta=\mu\left(1+\frac{d_{0}}{\alpha}\right)$, then $\textbf{z}=(x,y)\in\Omega$ and if $\beta>\mu\left(1+\frac{d_{0}}{\alpha}\right)$, then $\textbf{z}=(x,y)\in\Omega_1\bigcup\Omega_2.$
\end{rk}
\begin{proof} Let $\beta=\mu\left(1+\frac{d_{0}}{\alpha}\right).$ Then
$$\varphi(W_0(\textbf{z}))=\mu x'+\beta y'=\mu x+\beta y-\frac{d_0\mu x^2}{1+x}=\varphi(\textbf{z})-\frac{d_0\mu x^2}{1+x}.$$
Thus $\varphi(W_0(\textbf{z}))\leq\varphi(\textbf{z})$ for all $\textbf{z}=(x,y)\in\Omega$, that is the function $\varphi(\textbf{z})$ is a Lyapunov function for the operator $W_0$.

Let $\beta>\mu\left(1+\frac{d_{0}}{\alpha}\right).$
Let's look at the expression.
$$\varphi(W_0(\textbf{z}))-\varphi(\textbf{z})=\mu x'+\beta y'-(\mu x+\beta y)=(\beta-\mu)\frac{\alpha x}{1+x}-d_0\mu x=$$
$$=\frac{d_0\mu x}{1+x}\left(\frac{\alpha(\beta-\mu)}{d_0\mu}-1-x\right)=\frac{d_0\mu x}{1+x}\left(x^*-x\right).$$
If $\textbf{z}=(x,y)\in\Omega_1$, then $\varphi(W_0(\textbf{z}))\geq\varphi(\textbf{z})$, if $\textbf{z}=(x,y)\in\Omega_2$, then $\varphi\left(W_0(\textbf{z})\right)\leq\varphi(\textbf{z}).$

Hence, the function $\varphi(\textbf{z})$ is increasing on the set $\Omega_1$ and decreasing on the set $\Omega_2$. The function $\varphi(\textbf{z})$ is a Lyapunov function for the operator $W_0$.
\end{proof}

\subsection{The $\omega$-limit set}\

The problem of describing the $\omega$-limit set of a trajectory is of great importance in the theory of dynamical systems.

The following theorem  describes the trajectory of any initial point $\left(x^{(0)}, y^{(0)}\right)$ in $\Omega$.
\begin{thm}\label{omega} For the operator $W_{0}$ given by (\ref{syst}) (i.e. under condition (\ref{parametr})) the following hold:
\begin{itemize}
\item[(i)] If $\beta=\mu\left(1+\frac{d_{0}}{\alpha}\right)$, then for any initial point $\left(x^{(0)}, y^{(0)}\right)\in\Omega\setminus Fix W_0$  $$\lim\limits_{n\to \infty}x^{(n)}=0, \ \lim\limits_{n\to \infty}y^{(n)}=0;$$
\item[(ii)] If $\beta>\mu\left(1+\frac{d_{0}}{\alpha}\right)$, then for any initial point $\left(x^{(0)}, y^{(0)}\right)\in\Omega\setminus Fix W_0$ $$\lim\limits_{n\to \infty}x^{(n)}=x^*, \ \lim\limits_{n\to \infty}y^{(n)}=y^*;$$
\end{itemize}
where $\left(x^{(n)}, y^{(n)}\right)=W_0^n\left(x^{(0)}, y^{(0)}\right)$, with $W_{0}^n$ is $n$-th iteration of $W_{0}$  and $(x^*,y^*)$ is the fixed point defined by $(\ref{x*y*}).$
\end{thm}
\begin{proof} We have
\begin{equation}\label{x^ny^n} x^{(n)}=\beta y^{(n-1)}-\frac{\alpha x^{(n-1)}}{1+x^{(n-1)}}-d_{0}x^{(n-1)}+x^{(n-1)},\ y^{(n)}=\frac{\alpha x^{(n-1)}}{1+x^{(n-1)}}-\mu y^{(n-1)}+y^{(n-1)}.\end{equation}

First, we prove the assertion $(i).$
Let $\beta=\mu\left(1+\frac{d_{0}}{\alpha}\right).$ The function $\varphi(\textbf{z})$ is a decreasing Lyapunov function by Proposition (\ref{Lpro}) and bounded from below it follows that
existence of the $\lim\limits_{n\rightarrow\infty}\varphi(\textbf{z}^{(n)}).$
$$\lim\limits_{n\rightarrow\infty}\varphi(\textbf{z}^{(n)})=\lim\limits_{n\rightarrow\infty}\left(\varphi(\textbf{z}^{(n-1)})+\frac{\alpha(\beta-\mu)\left(x^{(n-1)}\right)^2}{1+x^{(n-1)}}\right)= \lim\limits_{n\rightarrow\infty}\varphi(\textbf{z}^{(n-1)})+$$
$$+\alpha(\beta-\mu)\lim\limits_{n\rightarrow\infty}\frac{\left(x^{(n-1)}\right)^2}{1+x^{(n-1)}},\ \Rightarrow \ \lim\limits_{n\rightarrow\infty}\frac{\left(x^{(n-1)}\right)^2}{1+x^{(n-1)}}=0, \ \Rightarrow \ \lim\limits_{n\rightarrow\infty}x^{(n-1)}=0.$$

By (\ref{x^ny^n}) and $\lim\limits_{n\rightarrow\infty}x^{(n)}=0$ we have $\lim\limits_{n\rightarrow\infty}y^{(n)}=0.$

Let's prove the assertion $(ii)$. We consider the trajectory of any initial point $\left(x^{(0)}, y^{(0)}\right)$ in the sets $\Omega_1$ and $\Omega_2$ . The function $\varphi(\textbf{z})$ is increasing on the set $\Omega_1$ and is bounded from above. The function $\varphi(\textbf{z})$ is decreasing in the set $\Omega_2$ and is bounded from below. So, the existence of the
\begin{equation}\label{varfi}\lim\limits_{n\rightarrow\infty}\varphi(\textbf{z}^{(n)})\neq0.\end{equation}
$$\lim\limits_{n\rightarrow\infty}\varphi(\textbf{z}^{(n)})=\lim\limits_{n\rightarrow\infty}\left(\varphi(\textbf{z}^{(n-1)})+\frac{d_0\mu x^{(n-1)}}{1+x^{(n-1)}}\left(x^*-x^{(n-1)}\right)\right)= \lim\limits_{n\rightarrow\infty}\varphi(\textbf{z}^{(n-1)})+$$
$$+d_0\mu\lim\limits_{n\rightarrow\infty}\frac{x^{(n-1)}}{1+x^{(n-1)}}\left(x^*-x^{(n-1)}\right),\ \Rightarrow \ \lim\limits_{n\rightarrow\infty}\frac{x^{(n-1)}}{1+x^{(n-1)}}\left(x^*-x^{(n-1)}\right)=0,\ \Rightarrow$$ $$\Rightarrow \ \lim\limits_{n\rightarrow\infty}x^{(n-1)}=0,\ or \ \lim\limits_{n\rightarrow\infty}x^{(n-1)}=x^*.$$
Assume $\lim\limits_{n\rightarrow\infty}x^{(n)}=0.$ Then, from (\ref{x^ny^n}) we have $\lim\limits_{n\rightarrow\infty}y^{(n)}=0.$ Hence, $\lim\limits_{n\rightarrow\infty}\varphi(\textbf{z}^{(n)})=0.$
This is a contradiction to (\ref{varfi}). So, $\lim\limits_{n\rightarrow\infty}x^{(n)}=x^*.$

By (\ref{x^ny^n}) and $\lim\limits_{n\rightarrow\infty}x^{(n)}=x^*$ we have $\lim\limits_{n\rightarrow\infty}y^{(n)}=y^*.$

Now, we consider the trajectory of an initial point $\left(x^{(0)}, y^{(0)}\right)$ in the sets $\Omega_3$ and $\Omega_4$.

Suppose $(x^{(n)}, y^{(n)})\in\Omega_3$ for all $n\in N$  and for any initial point $(x^{(0)}, y^{(0)})$ taken from the set $\Omega_3$. Then the sequence $x^{(n)}$ is decreasing and bounded from below, the sequence $y^{(n)}$ is increasing and bounded from above.
Indeed, from $x^*< x^{(n)}\leq\frac{\alpha\beta}{\mu d_{0}},\ 0\leq y^{(n)}\leq y^*$ we get:
$$x^{(n+1)}-x^{(n)}=\beta y^{(n)}-\frac{\alpha x^{(n)}}{1+x^{(n)}}-d_0x^{(n)}\leq \beta y^*-\frac{\alpha x^*}{1+x^*}-d_0x^*=0,$$
$$y^{(n+1)}-y^{(n)}=\frac{\alpha x^{(n)}}{1+x^{(n)}}-\mu y^{(n)}\geq\frac{\alpha x^*}{1+x^*}-\mu y^*=0.$$
Hence, the sequences $x^{(n)}, y^{(n)}$ have their limits. From (\ref{x^ny^n}) it follows that $\lim\limits_{n\rightarrow\infty}x^{(n)}=x^*, \lim\limits_{n\rightarrow\infty}y^{(n)}=y^*.$

 Assume $(x^{(n)}, y^{(n)})\in\Omega_4$ for all $n\in N$  and for any initial point $(x^{(0)}, y^{(0)})$ taken from the set $\Omega_4$. Then the sequence $x^{(n)}$ is increasing and bounded from above, the sequence $y^{(n)}$ is decreasing and bounded from below.
Indeed, from $0\leq x^{(n)}<x^*, y^*<y^{(n)}\leq\frac{\alpha}{\mu}$ we get:
$$x^{(n+1)}-x^{(n)}=\beta y^{(n)}-\frac{\alpha x^{(n)}}{1+x^{(n)}}-d_0x^{(n)}\geq\beta y^*-\frac{\alpha x^*}{1+x^*}-d_0x^*=0,$$
$$y^{(n+1)}-y^{(n)}=\frac{\alpha x^{(n)}}{1+x^{(n)}}-\mu y^{(n)}\leq\frac{\alpha x^*}{1+x^*}-\mu y^*=0.$$
So, the sequences $x^{(n)}, y^{(n)}$ have their limits. Since (\ref{x^ny^n}), we have $\lim\limits_{n\rightarrow\infty}x^{(n)}=x^*, \lim\limits_{n\rightarrow\infty}y^{(n)}=y^*.$

\end{proof}

\section{Biological interpretations}

In biology an population biologist is interested in the long-term behavior of the population of a certain species or collection of species. Namely, what happens to an initial population of members. Does the population become arbitrarily large as time goes on? Does the population tend to zero, leading to extinction of the species? \cite{BR1} In this section we briefly give some answers to these questions related to our model of the mosquito population.

Each point (vector) $\textbf{z}=(x;y)\in \mathbb{R}^{2}_{+}$ can be considered as a state (a measure) of the mosquito population.

Let us give some interpretations of our main results (interpretation of  Theorem \ref{omega}):
\begin{itemize}
\item[(a)] Let $\beta=\mu\left(1+\frac{d_{0}}{\alpha}\right).$ Under this condition on $\beta$ (i.e. the birth rate of adults), the mosquito population dies;
\item[(b)] If the inequality $\beta>\mu\left(1+\frac{d_{0}}{\alpha}\right)$ holds for $\beta$, then of the mosquito population tends to the equilibrium state $(x^*, y^*)$ with the passage of time.
 \end{itemize}

\end{document}